\documentclass[11pt]{article}
\usepackage{a4wide,amsmath,amsbsy,amsfonts,amssymb,
stmaryrd,amsthm,mathrsfs,graphicx, amscd, tikz-cd, cancel, supertabular}
\usepackage[normalem]{ulem}
\usetikzlibrary{matrix,arrows,decorations.pathmorphing}

\usepackage{tikz}
\usetikzlibrary{matrix,arrows,decorations.pathmorphing}
\usepackage{qtree}

\usepackage[top=1.2in,bottom=1.2in,left=1.21in,right=1.21in]{geometry}
\usepackage{hyperref}
\newtheorem{theorem}{Theorem}[section]
\newtheorem{proposition}[theorem]{Proposition}

\newtheorem{lemma}[theorem]{Lemma}

\theoremstyle{definition}

\newtheorem{definition}[theorem]{Definition}
\newtheorem{example}[theorem]{Example}

\numberwithin{equation}{section}

\theoremstyle{definition}
\newtheorem{remark}[theorem]{Remark}

\newcommand{\Ac}{\mathcal{A}}

\newcommand{\Mc}{\mathcal{M}}
\newcommand{\Rc}{\mathcal{R}}

\newcommand{\Oc}{\mathcal{O}}

\newcommand{\Cb}{\mathbb{C}}

\newcommand{\Fb}{\mathbb{F}}

\newcommand{\Pb}{\mathbb{P}}

\newcommand{\Rb}{\mathbb{R}}

\newcommand{\Zb}{\mathbb{Z}}
\newcommand{\Z}{\mathbb{Z}}

\newcommand{\xdownarrow}[1]{%
  {\left\downarrow\vbox to #1{}\right.\kern-\nulldelimiterspace}
} % \downarrow \big\downarrow \Big\downarrow \bigg\downarrow \Bigg\downarrow \xdownarrow{2cm}

\newcommand{\beq}{\begin{eqnarray}}
\newcommand{\eeq}{\end{eqnarray}}

\newcommand{\para}[1]{\medskip\noindent\textbf{#1.}}

\DeclareMathOperator{\Aut}{Aut}

\DeclareMathOperator{\GL}{GL}

\DeclareMathOperator{\Hom}{Hom}

\DeclareMathOperator{\Sp}{Sp}
\DeclareMathOperator{\Teich}{Teich}

\DeclareMathOperator{\Diff}{Diff}
\DeclareMathOperator{\Mod}{Mod}

\DeclareMathOperator{\J}{J}
\DeclareMathOperator{\vol}{vol}
\DeclareMathOperator{\Hol}{Hol}
\DeclareMathOperator{\Mor}{Mor}
\DeclareMathOperator{\DM}{\overline{\Mc}_g^{\rm DM}}

\title{Global rigidity of the period mapping}
\author{Benson Farb\thanks{Supported in part by National Science Foundation Grant No. DMS-181772, the Eckhardt Faculty Fund and the Jump Trading Mathlab Research Fund. }}

\begin{document}
\maketitle

%\begin{abstract}
%Let $\Mc_{g,n}$ denote the moduli space of smooth, genus $g\geq 1$ curves with $n\geq 0$ marked points.  Let $\Ac_h$ denote the moduli space of $h$-dimensional, principally polarized abelian varieties.  Let $g\geq 4$ and $h\leq g$.  If $F:\Mc_{g,n}\to\Ac_h$ is a nonconstant holomorphic map then $h=g$ and $F$ is the classical period mapping, assigning to a Riemann surface $X$ its Jacobian.  \end{abstract}
%
%\tableofcontents

\section{Introduction}

Let $\Mc_g$ denote the moduli space of smooth, genus $g\geq 1$ curves and let $\Ac_h$ denote the moduli space of $h$-dimensional, principally polarized abelian varieties.  
Let \[\J:\Mc_g\to\Ac_g\] denote the {\em period mapping}, assigning to a Riemann surface $X\in\Mc_g$ its Jacobian variety together with the prinicipal polarization induced from the intersection form on $H_1(X;\Zb)$.  The map $\J$ is an injective (by the Torelli Theorem) morphism of quasiprojective varieties.

Given the more than 150 years of intensive study of the period mapping, and the fundamental role it plays in the theory of Riemann surfaces, it is natural to ask:  
are there other ways to attach an $h$-dimensional principally polarized abelian variety to a smooth, genus $g$ Riemann surface in an algebraically (or even holomorphically) varying manner, not necessarily in a one-to-one fashion?   The main theorem of this paper states that, when $h\leq g$, the period mapping is the unique nontrivial way to do this, even if one is 
allowed the extra data of a finite set of marked points on the surface.  

Let $\Mc_{g,n}$ denote the moduli space of pairs $(X,(z_1,\ldots,z_n))$ with $X\in \Mc_g=\Mc_{g,0}$ and $z_i\neq z_j\in X$ when $i\neq j$.   In this case the period mapping $\J:\Mc_{g,n}\to\Ac_g$ factors through the map $\Mc_{g,n}\to\Mc_g$ given by $(X,(z_1,\ldots,z_n))\mapsto X$.
 
\begin{theorem}[{\bf Global rigidity of the period mapping}]
\label{theorem:main}
Let $g\geq 3, n\geq 0$ and assume that $h\leq g$.  Let $F:\Mc_{g,n}\to\Ac_h$ be any nonconstant holomorphic map of complex orbifolds.  Then $h=g$ and $F=\J$.
\end{theorem}

\begin{remark}
Both $\Mc_{g,n}$ and $\Ac_g$ are complex orbifolds: they are quotients of contractible complex manifolds by a group of biholomorphic automorphisms acting properly discontinuously and virtually freely.  As such, in this paper holomorphic maps between them are always taken to be in the category of orbifolds.  The standard examples of maps between orbifold moduli spaces (including the period mapping $\J$) are holomorphic in this sense. See \S\ref{section:proof} for more details.
\end{remark}

I do not know if the statement of Theorem~\ref{theorem:main} holds for $g=2$. Some upper bound on $h$ in terms of $g$ is necessary for Theorem~\ref{theorem:main} to hold: 
lifting complex structures to certain characteristic covers (e.g.\ finite homology covers) gives various nonconstant holomorphic maps $\Mc_g\to \Ac_h$ with $h>g$.  The following example shows that if $\Mc_g$ is replaced by a finite cover then the conclusion of Theorem~\ref{theorem:main} no longer holds.  

\begin{example}[{\bf Prym map}]
\label{remark:Prym}
Let 
\[\Rc_g:=\{(X,\theta) : X\in \Mc_g, \ 0\neq\theta\in H^1(X;\Fb_2)\}, \]
a $(2^{2g}-1)$-sheeted cover of $\Mc_g$. Then, in addition to the composition 
$\Rc_g\to \Mc_g\stackrel{\J}{\to}\Ac_g$, there is for $g\geq 2$ a nontrivial morphism of varieties 
\[{\rm Prym}:\Rc_g\to\Ac_{g-1}\]
called the {\em Prym map};  see for example the survey \cite{F} by G. Farkas.  There are many other such examples. 
\end{example}

I believe that, with the methods of this paper, it should be possible to prove that if $F$ is any nonconstant holomorphic map $\Rc_g\to \Ac_h$ with $h\leq g-1$ then $h=g-1$ and $F={\rm Prym}$.\footnote{Since an earlier version of this paper appeared, C. Servan \cite{Se} has proven this result using the methods of this paper.}  As Hain observed to me, post-composing the period mapping with Hecke correspondences on $\Ac_g$ gives a huge number of distinct holomorphic maps from various finite covers $X$ of $\Mc_g$ to $\Ac_g$.  A classification of all such maps $X\to \Ac_g$ seems a worthwhile challenge.

\para{Proof outline} 
The proof of Theorem~\ref{theorem:main} is divided into six steps, outlined as follows. As far as I know, this approach to proving uniqueness of holomorphic maps in a homotopy class - by using a convexity argument along Jacobi fields to 
convert a homotopy to an algebraic deformation, and play this off against rigidity for certain families -  is novel.   
\begin{enumerate}
\item A theorem of Korkmaz \cite{K} and the Congruence Subgroup Property classify 
representations $\pi^{\rm orb}_1(\Mc_{g,n})\to \pi^{\rm orb}_1(\Ac_h)$. The fact that $\Ac_h$ is aspherical quickly reduces the theorem to the case $h=g$ and $F$ homotopic to $\J$. 

\item Since $\Ac_g$ is covered by a bounded symmetric domain in $\Cb^{\binom{g+1}{2}}$, a theorem of Borel-Narasimhan \cite{BN} gives $F=\J$ as long as $F(x)=\J(x)$ for some $x\in\Mc_g$. The rest of the proof of Theorem~\ref{theorem:main} is devoted to finding such an $x$.

\item For any curve $C\subset\Mc_g$, we replace the homotopy $\J|_C\sim F|_C$ by a geodesic homotopy $H_t$.  We then follow an argument of Antonakoudis-Aramayona-Souto \cite{AAS}.  Using the fact that $\Ac_g$ has nonpositive sectional curvature, we 
deduce convexity of the energy of $H_t$ along a Jacobi field.  We apply this to a Wirtinger-type inequality, used as a kind of ``holomorphicity detector'', to prove that each $H_t:C\to \Ac_g$ is holomorphic. 

\item We apply a theorem of Kobayashi-Ochiai \cite{KO} to extend each $H_t$ to a map $\overline{C}\to\overline{\Ac_g}^{\rm S}$ to the Satake compactification of $\Ac_g$. Chow's Theorem gives that each $H_t$ is a {\em morphism} of varieties.  This improvement is needed to apply the rigidity theory of Faltings and Saito for families of abelian varieties.

\item We find a curve $C\subset\Mc_g$ that is {\em $\Ac_g$-rigid}: it is an isolated point in the space $\Mor(C,\Ac_g)$ of morphisms $C\to\Ac_g$.  The construction is non-explicit, and uses $g\geq 3$.  It produces a family of abelian varieties over $C$ with monodromy $\Sp(2g,\Zb)$, and another property, so that a rigidity criterion of Saito \cite{S} (building on Faltings) can be applied.  

\item Steps 4 and 5 give that the path $t\mapsto H_t$ in $\Mor(C,\Ac_g)$ is constant, from which it follows that $F(x)=\J(x)$ for all $x\in C$.  We have thus found the required (by Step 2) $x$, completing the proof of the theorem.
\end{enumerate}

An alternative approach to proving Step 3, due to R. Hain, is given in \S\ref{section:Hodge}. 

\para{Related work} There are various powerful theorems of Faltings, Noguchi and others 
\footnote{For a recent example, see Theorem 1.7 of Javanpeykar-Litt \cite{JL}.} 
giving the finiteness 
of various spaces of holomorphic maps between (typically compact) complex spaces and certain Kobayashi hyperbolic spaces such as $\Ac_g$; see for example the survey \cite{No} and \cite{JL}.  This contrasts with Theorem~\ref{theorem:main} above, where this finite set is shown to have a single element.  This is similar to the difference between local rigidity for cocompact, irreducible lattices in semisimple Lie groups, proved by Calabi-Vasenti-Weil, and the global rigidity later proved by Mostow.  It would be interesting to see if any of the methods in the present paper could be used to address any of the conjectures in \cite{No}.

\para{Acknowledgements}  I would like to thank Juan Souto for patiently answering my questions about \cite{AAS}; Sam Grushevsky for suggesting 
Example~\ref{remark:Prym}; Nir Gadish, Ariyan Javapeykar, Eduard Looijenga and Dan Margalit for useful comments and corrections on an earlier draft; Richard Hain for suggesting the comment above as well an alternative proof of Step 3; and Curtis McMullen for various inspiring conversations, and for useful comments on an earlier draft of this paper.  Finally, it is a pleasure to thank the anonymous referee for several useful suggestions.

\section{Proof of Theorem~\ref{theorem:main}}
\label{section:proof}

The purpose of this section is to prove Theorem~\ref{theorem:main}.  

\begin{remark}[{\bf Maps of orbifolds}]
A topological (resp.\ complex) {\em orbifold} is the quotient $X/\Gamma$ of a manifold (resp.\ complex manifold) $X$ by a group $\Gamma$ acting properly discontinuously on $X$ by homeomorphisms (resp.\ biholomorphic automorphisms).  Let $Y/\Lambda$ be another orbifold, and let $\rho:\Gamma\to\Lambda$ be a homomorphism.  A continuous (resp.\ holomorphic) {\em map in the category of orbifolds}  $F:X/\Gamma\to Y/\Lambda$ is by definition a continuous (resp.\ holomorphic) map $\widetilde{F}:X\to Y$ that intertwines $\rho$: 
\[\widetilde{F}(g\cdot x)=\rho(g)(\widetilde{F}(x)) \ \ \ \text{for all\ } x\in X, g\in \Gamma.\]
When this is the case we use the shorthand $F_*:\Gamma\to\Lambda$.   Note that this homomorphism is only defined up to postcomposition by an inner automorphism of $\Lambda$.  Henceforth all maps between orbifolds will be assumed to be maps in the category of orbifolds.
\end{remark}

We now prove Theorem~\ref{theorem:main} in 6 steps.

\subsection*{Step 1: Homotopy classes of maps $\Mc_{g,n}\to\Ac_h$}

The main goal of this step is the following.

\begin{proposition}
\label{proposition:homotopy1}
Let $g\geq 3, n\geq 0, h\geq 1$.  Assume that $h\leq g$.  Let $F:\Mc_{g,n}\to\Ac_h$ be any continuous map (in the category of orbifolds - see below).  Then either $F$ is homotopically trivial or $h=g$ and $F$ is homotopic to the period mapping $\J$ (and so in particular if $n>0$ then $F$ factors through the forgetful map $\Mc_{g,n}\to\Mc_g$).
\end{proposition}

The proof of Proposition~\ref{proposition:homotopy1} uses in a crucial way a result of Korkmaz (\cite{K}, Theorem 1) classifying low-dimensional representations of $\Mod(S_{g,n})$.

\begin{proof}[{\bf Proof of Proposition~\ref{proposition:homotopy1}}]
For now fix $g\geq 1, n\geq 0$.   Let $\Mod(S_{g,n}):=\pi_0(\Diff^+(S_g; z_1,\ldots ,z_n))$ be the {\em mapping class group}  of a smooth, oriented genus $g$ surface $S_g$ fixing $n$ distinct points on $S_g$, and let $\Sp(2g,\Zb)$ denote the integral symplectic group.  Let $\Teich(S_{g,n})$ be the Teichm\"{u}ller space of isotopy classes of complex structures on a smooth, genus $g$ surface $S_g$ with $n$ marked points.  Let ${\mathfrak h}_g$ denote the Siegel upper half-space.  The complex manifold $\Teich(S_{g,n})$ (resp.\ ${\mathfrak h}_g$) is known to be a bounded domain in $\Cb^N$ for $N=3g-3+n$ (resp.\ $N=\binom{g+1}{2}$).   

The group $\Mod(S_{g,n})$ acts on $\Teich(S_{g,n})$ properly discontinuously by biholomorphic automorphisms, and the quotient $\Mc_{g,n}:=\Teich(S_{g,n})/\Mod(S_{g,n})$ is the moduli space of genus $g$ Riemann surfaces with $n$ marked (and ordered) points.  Similarly, the integral symplectic group $\Sp(2g,\Zb)$ acts properly discontinuously by biholomorphic automorphisms on Siegel upper half-space ${\mathfrak h}_g$, a bounded domain in $\Cb^{\binom{g+1}{2}}$.   The quotient $\Ac_g:={\mathfrak h}_g/\Sp(2g,\Zb)$ is the moduli space of principally polarized abelian varieties.  

The moduli spaces $\Mc_{g,n}$ and $\Ac_g$ are quasiprojective varieties, and the period mapping $\J:\Mc_{g,n}\to\Ac_g$ (as discussed in the introduction) is a morphism of complex varieties.   Let 
\[\rho:\Mod(S_{g,n})\to \Aut(H_1(S_g;\Zb),\hat{i})=\Sp(2g,\Zb)\]
be the {\em symplectic representation}; here $\hat{i}$ denotes the algebraic intersection form on $H_1(S_g;\Zb)$.  The morphism $\J$ lifts to a $\rho$-equivariant morphism $\widetilde{\J}:\Teich(S_g)\to{\mathfrak h}_g$, giving a commutative diagram  
  \[
\begin{array}{ccc}
\Teich(S_{g,n})&\stackrel{\widetilde{\J}}{\longrightarrow} & {\mathfrak h}_g\\
\Big\downarrow&&\Big\downarrow \\
\Mc_{g,n}&\stackrel{\J}{\longrightarrow}&\Ac_g
\end{array}
\]

Thus $\J$ is a holomorphic map (indeed morphism) in the category of complex orbifolds, with $\J_*=\rho$.    
  
As explained above, $F$ induces a homomorphism $F_*:\Mod(S_{g,n})\to \Sp(2h,\Zb)$.  Since $g\geq 3$, Theorem 1 of Korkmaz \cite{K} gives that either $F_*=0$ (the trivial homomorphism) or $h=g$ and $F_*(x)=A\rho(x)A^{-1}$ for some $A\in\GL(2g,\Cb)$.   
Assume the latter case.  As stated in Step 5 of the proof of Korkmaz's theorem, any nontrivial homomorphism $\Mod(S_{g,n})\to \Sp(2g,\Zb)$ with $g\geq 4$ factors through the standard symplectic representation $\rho$, inducing a representation $\tau:\Sp(2g,\Zb)\to\Sp(2g,\Zb)$ with $F_*=\tau\circ\rho$.   By the Congruence Subgroup Property for $\Sp(2g,\Zb)$ (see \cite{Me}, {\it Corollar 1} on p.128), any homomorphism $\tau$ with infinite image (as in our case, since $F_*(x)=A\rho(x)A^{-1}$) is an automorphism and is given by conjugation by some element of $\Sp(2g,\Zb)$.  It follows that $F_*:\Mod(S_{g,n})\to\Sp(2g,\Zb)$ is, after conjugation by some $A\in \Sp(2g,\Zb)$, equal to the symplectic representation $\rho$.  

Since ${\mathfrak h}_h$ is contractible, there is a bijection between continuous maps (as always, in the category of orbifolds) $F:\Mc_{g,n}\to\Ac_h$ and 
pairs $([\widetilde{F}],\tau)$ where $\tau:\Mod(S_{g,n})\to\Sp(2h,\Zb)$ is a homomorphism and $[\widetilde{F}]$ denotes equivariant homotopy classes of continuous maps $\widetilde{F}:\Teich(S_{g,n})\to{\mathfrak h}_h$ intertwining $\rho$.  Thus if $F_*=0$ then $F$ is freely homotopic to a constant map, and if $F_*=\rho$ then $F$ is freely homotopic to the period mapping $\J$.  
\end{proof}

Before continuing we dispense with the case where $F$ is homotopically trivial.

\begin{lemma}
\label{lemma:htrivial}
Let $g\geq 3, n\geq 0, h\geq 1$.  Let $F:\Mc_{g,n}\to\Ac_h$ be any holomorphic map of orbifolds.  If $F$ is homotopically trivial then $F$ is constant.
\end{lemma}

\begin{proof}
Since $F$ is homotopically trivial it lifts to a map $\widetilde{F}:\Mc_{g,n}\to {\mathfrak h}_h$, a bounded domain in $\Cb^N, N:=\binom{h+1}{2}$.  Let $\widetilde{F}_i:\Mc_{g,n}\to\Cb$ be the composition of $\widetilde{F}$ with the coordinate function 
$z_i$ on $\Cb^N$.  So $\widetilde{F}_i$ is a bounded holomorphic function on a smooth, quasiprojective variety, hence is constant (see, e.g.\ \cite{BN}, \S 1.2 (a)).  Since this is true for each $i$ it follows that $\widetilde{F}$ is constant, hence $F$ is constant.
\end{proof}

Proposition~\ref{proposition:homotopy1} and Lemma~\ref{lemma:htrivial} reduce the proof of Theorem~\ref{theorem:main} to the case when $h=g$, when $F$ factors through the forgetful morphism $\Mc_{g,n}\to\Mc_g$, and the resulting map $\Mc_g\to\Ac_g$ is homotopic to $\J$. In particular the statement of the theorem for $\Mc_{g,n}$ follows from that for $\Mc_g$.  It thus suffices to assume the following.

\begin{center}
{\it We henceforth assume that $h=g, n=0$ and that $F$ is homotopic to $\J$.} 
\end{center}

\subsection*{Step 2: The Borel-Narasimhan Theorem}

Homotopic holomorphic maps are not always equal, even for nonpositively curved Kahler targets.  Here is a simple example (there more serious examples, even with $\Ac_g$ target, due to Faltings and others; see \cite{S}.)

\begin{example}[{\bf Cautionary example}]
Let $X$ and $Y$ be connected, complex manifolds. For each fixed $x\in X$ let $F_x:Y\to X\times Y$ 
be defined by $F_x(y):=(x,y)$.   Then each $F_x$ is holomorphic, all the $F_x$ are homotopic to each other, and all the $F_x$ are distinct from each other.  Note that $F_{x_1}(Y)\cap F_{x_2}(Y)=\emptyset$ when $x_1\neq x_2$.
\end{example}

There is a sufficient criterion, proved by Borel-Narasimhan \cite{BN}, for homotopic holomorphic maps to be equal.  Recall that a function $\sigma:X\to [-\infty,\infty)$ on a complex manifold $X$ is {\em plurisubharmonic} (also called ``pseudo-convex'' in \cite{BN}) if it is upper semi-continuous, and if for every domain $\Omega\subseteq\Cb$ and every holomorphic map $\psi:\Omega\to X$, the function $\sigma\circ\psi$ is subharmonic on $\Omega$.  

\begin{theorem}[Borel-Narasimhan \cite{BN}, Theorem 3.6]
\label{theorem:BN3.6}
Let $X$ be a connected complex manifold that carries no non-constant plurisubharmonic 
function that is bounded above. Let $a\in X$ and let $Y$ be a complex manifold covered by a bounded domain in $\Cb^N$.  Let $u,v:X\to Y$ be two holomorphic maps such that $u(a)=v(a)$  and such that $u_*=v_*:\pi_1(X,a)\to \pi_1(M,u(a))$.  Then $u=v$.
\end{theorem}

Since $\Mc_g$ is a complex quasiprojective variety, Proposition 2.1 of \cite{BN} implies that any bounded plurisubharmonic function on $\Mc_g$ is constant.  Also, $\Ac_g$ has universal cover ${\mathfrak h}_g$, which is a bounded domain in $\Cb^{\binom{g+1}{2}}$.  We can thus apply Theorem~\ref{theorem:BN3.6} with $X=\Mc_g, Y=\Ac_g, u=F$ and $v=\J$.  Since $F$ and $\J$ are homotopic and holomorphic, we conclude the following.

\begin{lemma}
\label{lemma:BNconsequence}
To conclude that $F=\J$ it is enough to find some $x\in \Mc_g$ such that $F(x)=\J(x)$.
\end{lemma}

The rest of the proof of Theorem~\ref{theorem:main} is devoted to finding such an $x\in\Mc_g$.  

\subsection*{Step 3: The Wirtinger squeeze}

We continue with the running assumption that $F:\Mc_g\to\Ac_g$ is a holomorphic map homotopic to the period mapping $\J$.  This step is devoted to proving the following result. 

\begin{lemma}[{\bf Paths of holomorphic curves}]
\label{lemma:holopath}
Let $C\subset \Mc_g$ be any smooth (not necessarily projective) curve.  There exists a homotopy $H:[0,1]\times C\to\Ac_g$ with $H_0=\J|_C$ and $H_1=F|_C$ such that for each $t\in [0,1]$ the map 
\[H_t:C\to \Ac_g\] is holomorphic.
\end{lemma}

The proof of Lemma~\ref{lemma:holopath} follows closely the proof by Antonakoudis-Aramayona-Souto of the Imayoshi-Shiga Theorem; see 
\S 4 of \cite{AAS}.  

\begin{proof}[Proof of Lemma~\ref{lemma:holopath}]
Let $G:[0,1]\times C\to\Ac_g$ be the restriction of the given homotopy $F\sim \J$ to the curve $C$.  I claim that there is a homotopy $H:[0,1]\times C\to\Ac_g$ with the property that $H_0=F|_C, H_1=\J|_C$ and for each $x\in C$, each path $\beta_x:[0,1]\to\Ac_g$ defined by $\beta_x(t):=H_t(x)$ is a geodesic in $\Ac_g$.   To see this, for any $x\in C$ consider the path $\gamma_x(t):=G_t(x)$.   Lift this path to a path $\widetilde{\gamma}_x:[0,1]\to {\mathfrak h}_g$.  Recall that ${\mathfrak h}_g$ is a bounded symmetric domain whose $\Sp(2g,\Rb)$-invariant Kahler metric has nonpositive sectional curvature.  Thus there is a unique (not necessarily unit speed) geodesic $\beta_x:[0,1]\to {\mathfrak h}_g$ with $\beta_x(0)=\gamma_x(0)$ and 
$\beta_x(1)=\gamma_x(1)$.  Further, there is a canonical homotopy from $\gamma_x$ to $\beta_x$ given by orthogonal projection onto a geodesic segment.  Since orthogonal projection onto a geodesic segment varies continuously with its endpoints, we obtain after composition with the projection ${\mathfrak h}_g\to \Ac_g$ the claimed homotopy $H$.

Given any $x\in C$, since the path $t\mapsto H_t(x)$ is a geodesic it follows that for any $v\in T_xC$ the vector field 
$t\mapsto (D_xH_t)(v)$ is a Jacobi field along $\beta_x$ in $\Ac_g$.  Since $\Ac_g$ has nonpositive curvature,  the function $t\mapsto ||(D_xH_t)(v)||$ is convex; here the norm is taken with respect to the inner product on $D_xH_t(T_xC)$.

Let $f:X\to Y$ be any smooth map from a Kahler $1$-manifold to a 
Kahler manifold $Y$.  The {\em energy of $f$ at $x$} is defined to be
\[E_x(f):=\frac{1}{2} [||D_xf(u)||^2+||D_xf(v)||^2] 
\]
where $\{u,v\}$ is any orthonormal basis for $T_xX$ and where the norms are those induced by the inner product metric on $T_{f(x)}Y$ given by the Riemannian metric on $Y$.  The number $E_x(f)$ does not depend on the choice of $\{u,v\}$.   Since $t\mapsto ||(D_xH_t)(v)||$ is convex it follows that $t\mapsto E_x(t)$ is convex for each fixed $x$.   For any smooth $f:X\to Y$, the {\em energy of $f$} is defined to be 
\[E(f):=\int_XE_x(f)\vol_X\]
where $\vol_X$ is the volume form on $X$.  

The following result is a variation of the Wirtinger inequality.  Before stating it we remark that a complex structure on a genus $g\geq 0$ surface $X$ determines an orientation on $X$, and so an isomorphism $\wedge^2T_xX\to\Rb$ for each $x\in X$.  For each $x\in X$, 
the standard ordering $\leq$ on $\Rb$ pulls back via this isomorphism to an ordering on 
$\wedge^2T_x$.  Two $2$-forms on $X$ can thus be compared pointwise. 

\begin{proposition}[{\bf Eells-Sampson}]
\label{proposition:ES}
Let $f:X\to Y$ be a smooth map from a finite area Kahler $1$-manifold $X$ to a complete Kahler manifold $Y$.  Assume that $E(f)<\infty$.  Let $\omega_X$ and $\omega_Y$ be the Kahler forms on $X$ and $Y$, respectively.  
Then 
\begin{equation}
\label{eq:ES4}
f^*\omega_Y(x)\leq E_x(f)\omega_X(x) \ \ \ \text{for all}\ x\in X
\end{equation}
with equality if and only if $f$ is holomorphic at $x$.
\end{proposition}

\begin{proof}
For compact $X$ the inequality \eqref{eq:ES4} with both sides integrated over $X$ is stated as the second proposition on page 126 of \cite{ES}.  However, the proof there proceeds by proving \eqref{eq:ES4}.  The proof of \eqref{eq:ES4} in \cite{ES}, as well as the ``if and only if'' statement for equality, is by elementary {\em pointwise} estimates, as well as the assumption $E(f)<\infty$, which we are assuming.     
\end{proof}

The proof of Lemma 4.1 of \cite{AAS} with the target replaced by $\Ac_g$ with its Kahler metric {\em mutatis mutandis} gives that $E(H_t)<\infty$ for each $t\in [0,1]$.  We can thus apply Proposition~\ref{proposition:ES} to each map $H_t, t\in[0,1]$, giving 
\begin{equation}
\label{eq:squeeze0}
H_t^*\omega_{\Ac_g}(x)\leq E_x(H_t)\omega_C \ \ \ \text{for all}\ x\in C.
\end{equation}

Integrating \eqref{eq:squeeze0} pointwise gives 
\begin{equation}
\label{eq:squeeze1}
E(H_t)\geq \int_C(H_t^*\omega_{\Ac_g}).
\end{equation}

On the other hand, since $H_t$ is homotopic  to $H_s$ for all $s,t\in [0,1]$ it follows  \footnote{There is an issue when $C$ is not compact, but the argument {\it verbatim} on page 226 of \cite{AAS} gives this result.  It uses an exhaustion of $C$ by compact sets, together with Stokes's Theorem.} that
\begin{equation}
\label{eq:squeeze2}
 \int_C(H_s^*\omega_{\Ac_g})= \int_C(H_t^*\omega_{\Ac_g})\ \ \ \text{for all}\ s,t\in [0,1].
 \end{equation}
 
 Now, $H_0:=F$ and $H_1:=\J$ are holomorphic, so the equality statement of Proposition~ \ref{proposition:ES} together with \eqref{eq:squeeze2} implies that
 \begin{equation}
 \label{eq:squeeze3}
 E(H_0)=\int_C(H_0^*\omega_{\Ac_g})=\int_C(H_1^*\omega_{\Ac_g})=E(H_1).
 \end{equation}
 
Combining \eqref{eq:squeeze1}, \eqref{eq:squeeze2} and \eqref{eq:squeeze3} 
gives 
\begin{equation}
\label{eq:squeeze4}
E(H_t)\geq \int_C(H_t^*\omega_{\Ac_g})=\int_C(H_1^*\omega_{\Ac_g})=E(H_1)=E(H_0) \ \ \ \text{for each}\ t\in[0,1].
\end{equation}

Since $t\mapsto E(H_t)$ is convex, \eqref{eq:squeeze4}  implies that for each $t\in [0,1]$:
\begin{equation}
\label{eq:squeeze5}
E(H_t)=\int_C(H_t^*\omega_{\Ac_g}).
\end{equation}
Since the pointwise estimate \eqref{eq:squeeze0} holds for each $x\in C$, 
it follows from \eqref{eq:squeeze5} (and the fact that $E_x(H_t)\geq 0$) that equality holds in \eqref{eq:squeeze0} for all $x\in C, t\in[0,1]$.  Proposition~\ref{proposition:ES} then implies that for each fixed $t\in[0,1]$ the map $H_t:X\to\Ac_g$ is holomorphic.

\end{proof}
\subsection*{Step 4: Improvement to a path of morphisms}

The goal of this step is to prove that the holomorphic maps $H_t$ are in fact morphisms.

\begin{lemma}
\label{lemma:morphismpath}
Let $C\subset\Mc_g$ be any smooth algebraic curve.  For each $t\in [0,1]$ the map $H_t:C\to \Ac_g$ is a morphism of varieties.
\end{lemma}

\begin{proof}
We need the following.

\begin{proposition}[\cite{KO}, Theorem $2'$]
\label{proposition:KO1}
Let $X$ be a complex manifold and let $A\subset X$ be a locally closed complex submanifold.  Let $D$ be a bounded symmetric domain in some $\Cb^N$ and let $G$ be the largest connected subgroup of biholomorphic automorphisms of $D$.  (Thus $G$ is a semisimple Lie group of noncompact, Hermitian type.)  Let $\Gamma\subset G$ be a 
(not necessarily torsion-free) arithmetic subgroup of $G$, and let $Y:=D/\Gamma$.  
Let $\overline{Y}^S$ be the Satake compactification of $Y$.  Then every holomorphic mapping $X-A\to Y$ extends to a holomorphic mapping $X\to \overline{Y}^S$.  
\end{proposition}

Apply Proposition~\ref{proposition:KO1} with $X:=\overline{C}$, the projective closure of $C$; the set $A:=\overline{C}-C$, a (possibly empty) finite set of points; $D={\mathfrak h}_g$, the Siegel upper half-space, a bounded domain in $\Cb^{\binom{g+1}{2}}$; the group $G=\Aut({\mathfrak h}_g)=\Sp(2g,\Rb)$; the arithmetic group $\Gamma:=\Sp(2g,\Zb)$; 
the quotient $Y:=\Ac_g={\mathfrak h}_g/\Sp(2g,\Zb)$; and the holomorphic map 
$H_t:C\to\Ac_g$.   Note that the theorem in \cite{KO} is explicitly stated for the case when $\Gamma$ is not torsion free, as in the case $\Gamma=\Sp(2g,\Zb)$.  

Thus $H_t:C\to\Ac_g$ extends uniquely to a holomorphic map 
$\overline{H}_t: \overline{C}\to \overline{\Ac_g}^S$ where $\overline{\Ac_g}^S$ denotes the Satake compactification.  By Chow's Theorem applied to each fixed $H_t, t\in[0,1]$, the map $\overline{H}_t$ is algebraic; that is, it is a morphism of varieties.  It follows that the restriction $H_t:C\to\Ac_g$ to the Zariski open $C\subset\overline{C}$ is a morphism.
\end{proof}

\subsection*{Step 5: Existence of an $\Ac_g$-rigid curve in $\Mc_g$}

For complex varieties $X$ and $Y$ let $\Hol(X,Y)$ denote the space of holomorphic maps $X\to Y$ equipped with the compact-open topology.  It is known that $\Hol(X,Y)$ is a Zariski open subset of a compact complex space, but we will not need this.  The subset $\Mor(X,Y)\subseteq\Hol(X,Y)$ of morphisms $X\to Y$ inherits the subspace topology.   A morphism $\phi:X\to Y$ is {\em rigid} if it is an isolated point of $\Mor(X,Y)$.

\begin{definition}[{\bf $\Ac_g$-rigid curve}]
A curve $i:C\to\Mc_g$ is {\em $\Ac_g$-rigid} if 
\[\J\circ i:C\to\Ac_g\] is rigid; in other words, $\J\circ i$ is an isolated point of $\Mor(C,\Ac_g)$.  
\end{definition}

Faltings, Saito and others constructed many interesting nonrigid curves in $\Ac_g$; see \cite{S}.  M\"{o}ller \cite{Mo} found curves in $\Mc_g$ that are not $\Ac_g$-rigid.  In fact, while McMullen \cite{Mc} proved that Teichm\"{u}ller curves are rigid in $\Mc_g$, they are not $\Ac_g$-rigid.   I do not know any explicit curves in $\Mc_g$ that are $\Ac_g$-rigid.   However, such curves do exist.

\begin{proposition}
For all $g\geq 3$ there exists an $\Ac_g$-rigid curve $i:C\to\Mc_g$.
\end{proposition}

\begin{proof}
In this proof the orbifold nature of $\Mc_g$ will cause a complication, since the usual (i.e.\ not orbifold) fundamental group of $\Mc_g$ is trivial.  We will address this by excising the ``orbifold locus'' of $\Mc_g$, as follows.

Let $\Oc\subset \Mc_g$ be the subvariety of smooth, genus $g$ Riemann surfaces with nontrivial automorphism group; this is a union of irreducible components, one for each topological type of faithful finite group action on $S_g$.  The Riemann-Hurwitz Formula implies that the dimension (over $\Cb)$ of each irreducible component of $\Oc$ is at most $2g-1$, with equality precisely for the hyperelliptic locus in $\Mc_g$.  The space $\Mc_g-\Oc$ is a smooth, quasiprojective variety.

We first consider the case when $g\geq 4$.  We claim that in this case 
$\pi_1(\Mc_g-\Oc)\cong\Mod(S_g)$. To see this, let 
$p:\Teich(S_g)\to\Mc_g$ be the quotient of $\Teich(S_g)$ by the action of 
$\Mod(S_g)$.  Let $\widetilde{\Oc}:=p^{-1}(\Oc)$.  
The action of $\Mod(S_g)$ on $\Teich(S_g)$ restricts to an action of $\Mod(S_g)$ on $\Teich(S_g)-\widetilde{\Oc}$ that is free and properly discontinuous.  Since 
\[{\rm codim}_\Cb(\widetilde{\Oc})=(3g-3)-(2g-1)=g-2 \geq 2 \ \ \text{for $g\geq 4$}\]
it follows that 
\[\pi_1(\Teich(S_g)-\widetilde{\Oc})=\pi_1(\Teich(S_g))=0,\]
so that 
\[\pi_1(\Mc_g-\Oc)=\pi_1((\Teich(S_g)-\widetilde{\Oc})/\Mod(S_g))=\Mod(S_g).\]

Let $\DM$ denote the Deligne-Mumford compactification of $\Mc_g$.  It is a smooth projective variety $\DM\subset\Pb^N$ with normal crossing divisor.
Recall that $\dim_\Cb\DM=\dim_\Cb\Mc_g=3g-3$ for $g\geq 2$. Bertini's Theorem implies that for generic hyperplanes $P_1,\ldots ,P_{3g-4}$ in $\Pb^N$, the intersection 
\[C:=(\Mc_g-\Oc)\cap P_1\cap\cdots\cap P_{3g-4}\]
is a smooth curve in $\Mc_g-\Oc$. Let $i:C\to\Mc_g-\Oc$ denote the inclusion. The Lefschetz Hyperplane Theorem for quasiprojective varieties (see \cite{HT}) implies that 
\begin{equation}
\label{eq:monodromy5}
i_*:\pi_1(C)\to\pi_1(\Mc_g-\Oc)\cong\Mod(S_g)
\end{equation}
is a surjection.\footnote{Note that, had we not excised $\Oc$, then the above argument gives no information since $\pi_1(\Mc_g)=0$ for $g\geq 1$; here $\pi_1$ is the topological (not orbifold) fundamental group.}

The codimension $1$ strata of $\partial\Mc_g:=\DM-\Mc_g$ consist of limits of degenerations of complex curves 
in $\Mc_g$ that pinch some (topological type of) simple closed curve $\beta$ on $S_g$ to a point.  Let $Z\subset\partial\Mc_g$ denote the unique codimension-$1$ stratum corresponding to the case where the topological type of $\beta$ is non-separating.  Note that the monodromy of the universal family over $\Mc_g$ restricted to a small loop in $\Mc_g$ around $Z$ is the cyclic subgroup generated by a Dehn twist about a non-separating simple closed curve.

The pullback $(\J\circ i)^*$ of the universal principally polarized abelian variety over $\Ac_g$ gives an algebraic family $f:E\to C$ of $d$-dimensional, principally polarized abelian varieties over $C$.  The monodromy representation of this family is 
\[\mu:=(\J\circ i)_*:\pi_1(C)\to \Sp(2g,\Zb)\] where $\Sp(2g,\Zb)$  acts on $H_1(f^{-1}(c_0);\Zb)$ as the standard symplectic representation; 
here $c_0\in C$ is a basepoint for the family.  By \eqref{eq:monodromy5}, and since the action of $\Mod(S_g;\Zb)$ on $H_1(f^{-1}(c_0);\Zb)$ gives the standard symplectic representation, it follows that the monodromy $\mu:\pi_1(C)\to\Sp(2g,\Zb)$ of the family $f:E\to C$ is surjective.   In particular
\begin{equation}
\label{eq:saitocond1}
\text{\it The monodromy of the family $f:E\to C$ is irreducible. }
\end{equation}

Since $\dim_\Cb Z=\dim_\Cb\DM-1=3g-4$, it follows from Bezout's Theorem applied to the Zariski closure $\overline{C}$ in $\DM$ that $S:=\overline{C}\cap Z\neq\emptyset$.  Thus $C$ is a noncompact curve with punctures $\{s\in S\}$.  If $\gamma_s$ is a small loop around $s\in S$ then, as explained above,  $i_*([\gamma])\in\Mod(S_g)$ is a Dehn twist about a nonseparating curve.  Thus the monodromy $\mu([\gamma])\in\Sp(2g,\Zb)$ is a symplectic transvection.  In particular 

\begin{equation}
\label{eq:saitocond2}
\text{\it The monodromy $\mu([\gamma])\in\Sp(2g,\Zb)$ has infinite order.}
\end{equation}

The properties given in \eqref{eq:saitocond1} and \eqref{eq:saitocond2} are precisely the hypotheses of a criterion of Saito (\cite{S}, Theorem 8.6), which states that if these conditions hold then the family $f:E\to C$ is rigid; that is, the map $J\circ i:C\to \Ac_g$ 
is an isolated point in $\Mor(C,\Ac_g)$, finishing the proof of the $g\geq 4$ case.

We now consider the $g=3$ case (the proof actually works for all $g\leq 7$).  Saito (\cite{S}, Corollary 8.4) proved that any abelian scheme with no isotrivial factors and of relative dimension at most $7$ must be rigid.  Let $i:C\to\Mc_3$ be any curve with the property that there exists $c\in C$ such that $J\circ i(c)$ is a simple abelian variety; such $C$ are known to exist for all $g\geq 1$.  Applying Saito's result gives an $\Ac_g$-rigid curve in $\Mc_g$ for all $g\leq 7$, in particular for $g=3$.
\end{proof}

\subsection*{Step 6: Finishing the proof}

Let $C\subset \Mc_g$ be the $\Ac_g$-rigid curve constructed in Step 5.  By Steps 3 and 4, in particular Lemma~\ref{lemma:morphismpath}, the homotopy 
$G_t:\Mc_g\to\Ac_g$ between $J$ and $F$ restricts to a homotopy 
$H_t:C\to\Ac_g$, giving 
a path $\alpha:[0,1]\to\Mor(C,\Ac_g)$ defined by $\alpha(t):=H_t$.  
Since $C$ is $\Ac_g$-rigid the path component of $\alpha(0)=H_0=J|_C$ in the space $\Mor(C,\Ac_g)$ is a single point, and so the path $\alpha$ is constant.  In particular 
\[J(x)=H_0(x)=H_1(x)=F(x)\ \ \ \text{for all $x\in C$}.\]

By Lemma~\ref{lemma:BNconsequence}, this implies that $J(x)=F(x)$ for all $x\in\Mc_g$, proving Theorem~\ref{theorem:main}.

\section{An Alternative approach to Theorem~\ref{theorem:main}}
\label{section:Hodge}

In this section we give a different approach to proving Theorem~\ref{theorem:main} by replacing Step 3 with a Hodge theory argument.  This proof is due to Richard Hain, who informed us of it after reading an earlier draft of this paper.  For simplicity we consider only the case $n=0$.   

\begin{proof}
By Step 1, we are reduced to the case $h=g\geq 3$ and $F$ homotopic to $\J$.  By a slight variation of the argument in Step 5, we can cover 
$\Mc_g$ by complete curves with the property that the inclusion $i:C\to\Mc_g$ induces a surjection $i_*:\pi_1(C)\to\Mod(S_g)$.    Let $C$ be a curve in such a cover.  

By the same argument as in Step $4$, the restriction of $F$ to $C$ is in fact a morphism of varieties.  Let $V$ denote the standard polarized variation of Hodge structure (PVHS) over $A$, where the fiber over the PPAV $A$ is 
$H^1(A;\Zb)$.  Let $V_F$ (resp.\ $V_{\J}$) be the pullback of $V$ to $C$ via $F$ (resp.\ $\J$) restricted to $C$.  
Since $\pi_1(C)\to\Mod(S_g)$ is surjective, it follows that each composition $(\J\circ i)_*$ and $(F\circ i)_*$ is surjective on $\pi_1$, so that $V_F$ and $V_{\J}$ are simple PVHS.  It follows that
\[H^0(C;\Hom(V_F,V_{\J}))\cong\Z\]
as abelian groups.  Since $V_F$ and $V_{\J}$ have the same weight, it follows that this is a Hodge structure of weight $0$ and dimension $1$.  Thus it must be $\Zb(0)$.   The Theorem of the Fixed Part gives that for each $x\in C$ that
\[H^0(X,\Hom(V_F,V_{\J}))_x\to \Hom(V_F,V_{\J})_x\]
is a morphism of Hodge structures, so that $(V_F)_x\cong (V_{\J})_x$ is an isomorphism.   

Theorem 7.24 of Schmid \footnote{As Schmid remarks, this result with target a locally symmetric variety - which is the case here - already follows from Borel-Narasimhan \cite{BN}, as in our Step 2.} the quotient states that for two variations of Hodge $V,W$ structures over a base $C$ that is a Zariski open in a compact analytic space, and if there exists $x\in C$ with $V_x$ isomorphic to $W_x$, and if the isomorphism preserves the action of $\pi_1(C)$, then the isomorphism extends to an isomorphism of the variations of Hodge structure.  Applying this to our case gives that 
\begin{equation}
\label{eq:iso}
V_F\cong V_{\J} \ \ \text{as PVHS.}
\end{equation}

Now, if $\Ac_g$ were a classifying space for PVHS, then \eqref{eq:iso} would imply that $f=\J$.  However, this is not quite true since $\Ac_g$ is only a coarse moduli space for PPAV.  Thus one must repeat the above argument for $\Ac_g[3]$, the (fine!) moduli space of PPAV with level $3$ structure, conclude that the lifts of $F$ and $J$ agree on this cover, and intertwine the $\Sp(2g,\Fb_3)$-actions. We leave the details to the reader.
 \end{proof}

\bigskip{\noindent
Dept. of Mathematics, University of Chicago\\
E-mail: bensonfarb@gmail.com

\end{document}